\RequirePackage{fix-cm}
\documentclass[smallextended]{svjour3}       
\smartqed  
\usepackage{amsfonts,color,graphicx}
\begin{document}

\title{Generalized SOR iterative method for a class of complex symmetric
linear system of equations 
}


\author{Davod Khojasteh Salkuyeh \and Davod Hezari \and Vahid Edalatpour}

\authorrunning{D. K. Salkuyeh, D. Hezari and V. Edalatpour} 

\institute{D. K. Salkuyeh \at
              Faculty of Mathematical Sciences, University of Guilan, Rasht, Iran\\
              \email{khojasteh@guilan.ac.ir, salkuyeh@gmail.com}             \\[2mm]
D. Hezari \at
              Faculty of Mathematical Sciences, University of Guilan, Rasht, Iran\\
              \email{hezari\_h@yahoo.com}\\[2mm]
V. Edalatpour \at
              Faculty of Mathematical Sciences, University of Guilan, Rasht, Iran\\
              \email{vedalat.math@gmail.com}
}

\date{Received: date / Accepted: date}

\maketitle

\begin{abstract}

In this paper, to solve a broad class of complex symmetric linear systems, we recast the complex system in a real formulation and apply the generalized successive overrelaxation (GSOR) iterative method to the equivalent real system. We then investigate its convergence properties and determine its optimal iteration parameter as well as its corresponding optimal convergence factor. In addition, the resulting GSOR preconditioner is used to preconditioned Krylov subspace methods such as GMRES for solving the real equivalent formulation of the system. Finally, we give some numerical experiments to validate the theoretical results and compare the performance of the GSOR method  with the modified Hermitian and skew-Hermitian splitting (MHSS) iteration.

\end{abstract}

\section{Introduction}
\label{SEC1}

Consider the system of linear equations
\begin{equation}\label{1.1}
Au= b, \qquad  A\in\mathbb{C}^{n\times n},\quad
u,b\in\mathbb{C}^{n},
\end{equation}
where $A$ is a complex symmetric matrix of the form
\begin{equation}\label{1.2}
A=W+iT, \hspace{.3cm} (i=\sqrt{-1})
\end{equation}
and $W,T\in\mathbb{R}^{n\times n}$ are symmetric matrices with at least one of them being positive
definite. Here, we mention that the assumptions of  \cite{Bai-G,MHSS} are stronger than ours, where the author assumed that both of the matrices $W$ and $T$ are symmetric positive semidefinite matrices with at least one of them being positive definite.
Hereafter, without lose of generality, we assume that $W$ is symmetric positive definite. Such systems arise in many problems in scientific
computing and engineering such as FFT-based solution of certain time-dependent PDEs \cite{Bertaccini}, structural dynamics
\cite{Feriani}, diffuse optical tomography \cite{Arridge}, quantum mechanics \cite{Dijk} and molecular scattering
\cite{Poirier}. For more applications of this class of complex symmetric systems, the reader is referred to \cite{Benzi-B} and  references
therein.

Bai et al. in \cite{Bai-G}  presented the Hermitian/skew-Hermitian splitting (HSS) method to solve non-Hermitian positive definite system of linear equations. After that, this method gains people's attention and proposed different variants of the method. Benzi and Gloub in \cite{Benzi-Golub-SIMAX} and Bai et al. in \cite{Bai-Golub-Pan} have applied the HSS method to solve saddle point problem or as a precondtioner. The normal/skew-Hermitian splitting (NSS) method has been presented by Bai et al. in \cite{Bai-Golub-Ng}. Moreover, Bai et al. \cite{PSS} have presented the positive definite and skew-Hermitian splitting (PSS) method to solve positive definite system of linear equations. Lopsided version of the HSS (LHSS) method has been presented by Li et al. in \cite{LHSS}. More recently Benzi in \cite{GHSS} proposed a generalization of the HSS method to solve positive definite system of linear equation.

We observe that the matrix $A$ naturally possesses a Hermitian/skew-Hermitian (HS) splitting
\begin{equation}\label{1.3}
A=H+S,
\end{equation}
where
\[
H=\frac{1}{2}(A+A^*)=W \quad {\rm and} \quad S=\frac{1}{2}(A-A^*)=iT,
\]
with $A^*$ being the conjugate transpose of $A$. Based on the HS splitting (\ref{1.3}), the HSS method to solve (\ref{1.1}) can be written as:

\medskip

\noindent{\it \textbf{The HSS iteration method}:}  \emph{Given an initial guess} $u^{(0)}
\in \mathbb{C}^{n}$ \emph{and positive constant}  $\alpha$, \emph{for} $k=0, 1, 2 \ldots,$ \emph{until} \{$u^{(k)}$\}
\emph{converges, compute}
\begin{equation}\label{1.4}
\left\{
  \begin{array}{ll}
    (\alpha I+W)u^{(k+\frac{1}{2})}=(\alpha I-iT)u^{(k)}+b, & \\
   (\alpha I+iT)u^{(k+1)}=(\alpha I-W)u^{(k+\frac{1}{2})}+b, &
  \end{array}
\right.
\end{equation}
\emph{where $I$ is identity matrix.}

\medskip

Since $W\in \mathbb{R}^{n \times n}$ is symmetric
positive definite, we know from \cite{Bai-G} that if $T$ is positive semidefinite then the HSS
iterative method is convergent for any positive constant $\alpha$.
In each iteration of the HSS method to solve (\ref{1.1}) two
sub-systems should be solved. Since  $\alpha I+W$ is  symmetric
positive definite, the first sub-system can be solved exactly by
the Cholesky factorization of the coefficient matrix and in the
inexact version, by the conjugate gradient (CG) method.  In the
second half-step of an iteration  we need to solve a shifted
skew-Hermitian sub-system in each iteration. This sub-system can
be solved by a variant of Bunch-Parlett factorization
\cite{Golub-V} or inexactly by a Krylov subspace iteration scheme
such as the GMRES method \cite{Saad,GMRES} which is the main problem of
HSS method.

Bai et al. \cite{MHSS} recently have presented the following
modified Hermitian and skew-Hermitian splitting (MHSS) method to
iteratively compute a reliable and accurate approximate solution
for the system of linear equations (\ref{1.1}). This method may be run as following:

\medskip

\noindent{\it \textbf{The MHSS iteration method:}}  \emph{Given an initial guess $u^{(0)}
\in \mathbb{C}^{n}$ and positive constant $\alpha$, for $k=0, 1, 2 \ldots,$ until \{$u^{(k)}$\}
converges, compute}
\begin{equation}\label{1.5}
\left \{\begin{array}{ll}
(\alpha I+W)u^{(k+\frac{1}{2})}=(\alpha I-iT)u^{(k)}+b,\\
(\alpha I+T)u^{(k+1)}=(\alpha I+iW)u^{(k+\frac{1}{2})}-ib.
\end{array}\right.
\end{equation}

\medskip

In \cite{MHSS}, it has been shown that if $T$ is symmetric positive semidefinite then the MHSS iterative method is convergent for any positive constant $\alpha$. Obviously both of the matrices $\alpha I+W$ and $\alpha I+T$ are symmetric positive definite. Therefore, the two sub-systems involved in each step of the MHSS
iteration can be solved effectively  by using the Cholesky factorization of the matrices  $\alpha I+W$ and $\alpha I+T$. Moreover, to solve both of the sub-systems  in the inexact variant of the MHSS method one can use a preconditioned conjugate gradient scheme. It is necessary to mention that the right hand
side of the sub-systems are still complex and the MHSS method involves complex arithmetic. Numerical results presented in \cite{MHSS} show that the MHSS method in general is more effective than the HSS, GMRES, GMRES(10) and GMRES(20) methods in terms of both iteration count and CPU time. More recently Bai et al. in \cite{PMHSS} proposed a preconditioned version of the MHSS method.

The well-known successive overrelaxation (SOR) method is a basic iterative method which is popular in engineering  and science applications. For example, it has been used to solve augmented linear systems \cite{SORlike,TSORlike,MSORlike} or as a preconditioner \cite{SORP1,SORP2}. In \cite{SSORlike}, Zheng et al. applied the symmetric SOR-like method to solve saddle point problems. In \cite{GSOR}, Bai et al. have presented a generalization of the SOR method to solve augmented linear systems.

In this paper, by equivalently  recasting the complex system of
linear equations (\ref{1.1}) in a  2-by-2 block real  linear
system, we define the generalized successive overrelaxation
(GSOR) iterative method to solve the equivalent real system.
Then, for the GSOR method, convergence conditions are derived and
determined its optimal iteration parameter and corresponding
optimal convergence factor. Besides its use as a solver, the GSOR
iteration is also used as a preconditioner to accelerate Krylov
subspace methods such as GMRES. Numerical experiments which use
GSOR as a preconditioner to GMRES, show a well-clustered spectrum
(away from zero) that usually lead to speedy convergence of the
preconditioned iteration. In the new method two sub-systems with
coefficient matrix $W$ should be solved which can be done by the
Cholesky factorization or inexactly by the CG algorithm.
Moreover, the right-hand side of the sub-systems are real.
Therefore, the solution of the system can be obtained by the real
version of the algorithms.

The rest of the paper is organized as follows. In Section 2 we
propose our method and investigate its convergence properties.
Section 3 is devoted to some numerical experiments to show the
effectiveness of the GSOR iteration method as well as the
corresponding GSOR preconditioner. Finally, in Section 4, some
concluding remarks are given .

\section{THE NEW METHOD} \label{SEC3}

Let $u=x+iy$ and  $b=p+iq$ where  $x,y,p,q\in \mathbb{R}^{n}$. In
this case the complex linear system (\ref{1.1}) can be written as
2-by-2 block real equivalent formulation
\begin{equation}\label{1.6}
\mathcal{A}\pmatrix{x \cr y}=\pmatrix{p \cr q},
\end{equation}
where
\[
\mathcal{A}=\pmatrix{W & -T \cr T & W}.
\]
We split the coefficient matrix of  (\ref{1.6})  as
\[
\mathcal{A}=\mathcal{D}-\mathcal{E}-\mathcal{F},
\]
where
\[
\mathcal{D}=\pmatrix{W & 0 \cr 0 & W},\quad  \mathcal{E}=\pmatrix{0 & 0 \cr -T & 0} \quad {\rm and}  \quad \mathcal{F}=\pmatrix{0 & T \cr 0 &
0}.
\]
In this case the GSOR iterative method to solve (\ref{1.6}) can
be written as
\begin{equation}\label{1.7}
\pmatrix{x^{k+1} \cr y^{k+1}}=\mathcal{G}_{\alpha}\pmatrix{x^{k} \cr
y^{k}}+ \alpha (\mathcal{D}-\alpha \mathcal{E})^{-1}\pmatrix{p \cr q},
\end{equation}
where $0\neq \alpha \in \mathbb{R}$ and
\begin{eqnarray*}
\mathcal{G}_{\alpha}&=&(\mathcal{D}-\alpha \mathcal{E})^{-1}((1-\alpha)\mathcal{D}+\alpha \mathcal{F})\\
&=& \pmatrix{W & 0 \cr \alpha T & W}^{-1}\pmatrix{(1-\alpha)W & \alpha T \cr 0 & (1-\alpha)W}\\
&=& \pmatrix{I & 0 \cr \alpha S & I}^{-1}\pmatrix{(1-\alpha)I & \alpha S \cr 0 & (1-\alpha)I},\\
\end{eqnarray*}
wherein  $S=W^{-1}T$. It is easy to see that (\ref{1.7}) is equivalent to
\begin{equation}\label{1.8}
\left \{\begin{array}{ll} Wx^{(k+1)}=(1-\alpha)Wx^{(k)} + \alpha Ty^{(k)}+ \alpha p,\\
Wy^{(k+1)}=-\alpha Tx^{(k+1)} + (1-\alpha)Wy ^{(k)}+ \alpha q,
\end{array}\right.
\end{equation}
where $x^{(0)}$ and $y^{(0)}$ are initial approximations for $x$ and  $y$, respectively. As we mentioned, the iterative method (\ref{1.8}) is real valued  and the coefficient matrix of both of the sub-systems is $W$.
Since the coefficient matrix ${\cal A}$ is a block two-by-two matrix, then one may imagine that the proposed method is a kind of the block-SOR method (see \cite{HadjiSOR,SongSOR,YoungSOR}). If we introduce
\[
\mathcal{M}_{\alpha}=\frac{1}{\alpha}(\mathcal{D}-\alpha \mathcal{E})\quad {\rm and} \quad \mathcal{N}_{\alpha}=\frac{1}{\alpha}((1-\alpha)\mathcal{D}+\alpha \mathcal{F}),
\]
then it holds that
\[
\mathcal{A}=\mathcal{M}_{\alpha}-\mathcal{N}_{\alpha}\quad {\rm and} \quad \mathcal{G}_\alpha=\mathcal{M}_{\alpha}^{-1} \mathcal{N}_{\alpha}.
\]
Therefore, GSOR is a stationary iterative method obtained by the matrix splitting
$\mathcal{A}=\mathcal{M}_{\alpha}-\mathcal{N}_{\alpha}$. Hence, we deduce that
the matrix $\mathcal{M}_\alpha$ can be used as preconditioner for the system (\ref{1.6}).
Note that, the multiplicative factor $1/\alpha$ can be dropped since it has no effect on
the preconditioned system. Therefore the preconditioned system takes the following form
\begin{equation}\label{1.9}
\mathcal{P}_{\alpha}^{-1}\mathcal{A}\pmatrix{x \cr y}=\mathcal{P}_{\alpha}^{-1}\pmatrix{p \cr q},
\end{equation}
where $\mathcal{P}_\alpha=\mathcal{D}-\alpha \mathcal{E}$. In the
sequel,  matrix $\mathcal{P}_\alpha$ will be referred to as the
GSOR preconditioner. Krylov subspace methods such as the GMRES
algorithm to solve (\ref{1.9}) involve only matrix-vector
multiplication of the form
\[
\pmatrix{e \cr f}=\mathcal{P}_{\alpha}^{-1}\mathcal{A}\pmatrix{r \cr s},
\]
which can be done in four steps by the following procedure
\begin{enumerate}
\item $t:=Wr-Ts$;
\item $u:=Tr+Ws$;
\item Solve $We=t$ for $e$;
\item Solve $Wf=u-\alpha Te$ for $f$.
\end{enumerate}
In the steps 3 and 4 of the above procedure one can use the Cholesky factorization of the matrix $W$.

In continuation, we investigate the convergence of the proposed method.
 \begin{lemma}\label{l1}
(\cite{Axelson}) Both roots of the real quadratic equation $x^{2}-rx+s=0$ are less than one in modulus if and only if $|s|<1$
and $|r|<1+s$.
 \end{lemma}
\begin{lemma}\label{l2}
Let $W,T \in \mathbb{R}^{n \times n} $ be symmetric positive definite and symmetric, respectively. Then, the eigenvalues of
the matrix $S=W^{-1}T$  are all real.
\end{lemma}
\begin{proof}
Since $W$ is a symmetric positive definite matrix, there is a symmetric positive definite matrix $R$ such that
 $W=R^{2}$ (see \cite{Golub-V}, page $149$). Therefore, we have
 \[
 RSR^{-1}=R^{-1}TR^{-1}=R^{-T}TR^{-1}:=Z.
 \]
This shows that $S$ is similar to $Z$. On the other hand $Z$ is symmetric and  therefore, the eigenvalues of $S$  are real. \qquad $\Box$
\end{proof}

The following theorem presents a necessary and sufficient condition for guaranteeing the convergence of the GSOR method.

 \begin{theorem}\label{t1}
Let $W,T \in \mathbb{R}^{n \times n} $ be symmetric positive definite and symmetric, respectively. Then, the GSOR method to solve Eq. (\ref{1.6}) is convergent if and only if
\[
0<\alpha<\frac{2}{1+\rho(S)},
\]
where $S=W^{-1}T$ and $\rho(S)$ is the spectral radius of $S$.
 \end{theorem}
 \begin{proof}
 Let $\lambda \neq 0$ be an eigenvalue of $\mathcal{G}_{\alpha}$ corresponding to  the eigenvector $z=(v^{T},w^{T})^T$. Note that for $\lambda = 0$
 there is nothing to investigate. Then, we have
\[
\pmatrix{(1-\alpha)I & \alpha S \cr 0 & (1-\alpha)I}\pmatrix{v
\cr w}  =\lambda\pmatrix{I & 0 \cr \alpha S & I}\pmatrix{v
\cr w},
 \]
or equivalently
\begin{equation}\label{1.10}
\left\{
  \begin{array}{ll}
   \alpha Sw=(\lambda+\alpha-1)v,  \\
   -(\lambda+\alpha-1)w=\lambda \alpha S v.
  \end{array}
\right.
\end{equation}
If $\lambda=1-\alpha$, for convergence of the GSOR method we must have $|1-\alpha|<1$, or equivalently
\begin{equation}\label{1.11}
  0<\alpha<2.
\end{equation}
If $\lambda\neq 1-\alpha$, from (\ref{1.10}), it is easy to verify that
\[
 (1-\alpha-\lambda)^{2}w=-\lambda\alpha^{2}S^{2}w.
\]
 This shows that for every eigenvalue  $\lambda \neq 0$  of  $\mathcal{G}_{\alpha}$  there is an eigenvalue $\mu$  of $S$  that satisfies
\begin{equation}\label{1.12}
  (1-\alpha-\lambda)^{2}=-\lambda\alpha^{2}\mu^{2}.
\end{equation}
Eq. (\ref{1.12}) is equivalent to
\begin{equation}\label{1.13}
  \lambda^{2}+(\alpha^{2}\mu^{2}+2\alpha-2)\lambda+(\alpha-1)^{2}=0.
\end{equation}
According to Lemma \ref{l1}, $|\lambda|<1$  if and only if
\begin{equation}\label{1.14}
\left \{\begin{array}{ll} |\alpha-1|^{2}<1,\\
|\alpha^{2}\mu^{2}+2\alpha-2|<1+(\alpha-1)^{2},
\end{array}\right.
\end{equation}
The first equation in (\ref{1.14}) is equivalent to (\ref{1.11}) and  the second equation in (\ref{1.14}) is equivalent to
$(\alpha-2)^{2}>\alpha^{2}\mu^{2}$. Therefore, according to Lemma \ref{l2}, the latter equation holds if and only if
\[
(\alpha-2)^{2}>\alpha^{2}\rho(S)^{2}.
\]
This relation is equivalent to $|\alpha-2|>\alpha\rho(S)$.
From (\ref{1.11}), it is easy to see  that the latter inequality
holds if and only if
\[
0<\alpha<\frac{2}{1+\rho(S)},
\]
which completes the proof. \qquad $\Box$
\end{proof}

In the next theorem, we obtain the optimal value of the relaxation parameter $\alpha$  which minimizes  the spectral radius
of the iteration matrix of the GSOR method, i.e.,
\[
\rho(\mathcal{G}_{\alpha^{*}})=\min_{0<\alpha<\frac{2}{1+\rho(S)}}\rho(\mathcal{G}_{\alpha}).
\]

\begin{theorem}\label{t2}
 Let $W,T \in \mathbb{R}^{n \times n} $ be symmetric positive definite and symmetric, respectively. Then, the optimal value of the relaxation parameter for the GSOR iterative method (\ref{1.8}) is given by
 \begin{equation}\label{1.15}
\alpha^*=\frac{2}{1+\sqrt{1+\rho^{2}(S)}},
\end{equation}
and the corresponding optimal convergence factor of the method is given by
\begin{equation}\label{1.16}
\rho(G_{\alpha^*})=1-\alpha^*=1-\frac{2}{1+\sqrt{1+\rho(S)^{2}}},
\end{equation}
where $S=W^{-1}T$ and $\rho(S)$ is the spectral radius of $S$.
\end{theorem}
 \begin{proof}
In the proof of Theorem \ref{t1}, we have seen that for every eigenvalue  $\lambda \neq 0$  of  $\mathcal{G}_{\alpha}$  there is an eigenvalue $\mu$  of $S$  that satisfies  (\ref{1.12}). We exploit the roots of this quadratic equation to determine
 the optimal parameter $\alpha^*$. The roots of Eq. (\ref{1.12}) are
\[
\lambda_{1,2}(\alpha)=\frac{-(\alpha^{2}\mu^{2}+2\alpha-2)\pm\sqrt{\Delta}}{2},
\]
where
\[
\Delta=\alpha^{2}\mu^{2}(\alpha^{2}\mu^{2}+4\alpha-4).
\]
From (\ref{1.12}), we can write
\begin{equation}\label{1.17}
\frac{\lambda+\alpha-1}{\alpha}=\pm\mu\sqrt{-\lambda}.
\end{equation}
We define the following functions
\[
f_{\alpha}(\lambda)=\frac{\lambda+\alpha-1}{\alpha} \quad {\rm and} \quad g(\lambda)=\pm \mu\sqrt{-\lambda}.
\]
Then $f_{\alpha}(\lambda)$ is a straight line through the point $(1,1)$, whose slope increases monotonically with
decreasing $\alpha$. It is clear that (\ref{1.17}) can be geometrically interpreted as the intersection of the curves
$f_{\alpha}({\lambda})$ and $g({\lambda})$, as illustrated in Fig. 1.
\begin{figure}[!hbp]
\centerline{\includegraphics[height=8cm,width=11cm]{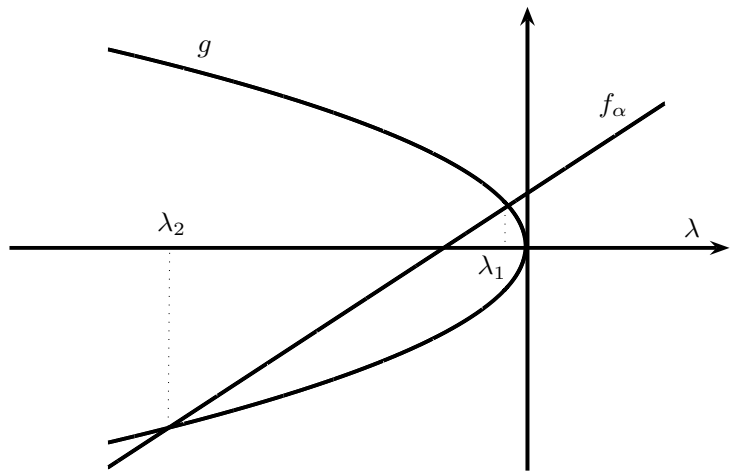}}
{\caption{Plot of the curves of $f_{\alpha}({\lambda})$ and $g({\lambda})$. }}\vspace{0cm}\label{figfg}
\end{figure}

Fig. 1 shows that when $\alpha$ decreases, it is  clear that the largest abscissa
of the two points of intersection decreases until $f_{\alpha}({\lambda})$ becomes tangent to
$g({\lambda})$. In this case, we have $\lambda_1=\lambda_2$ and as a result $\Delta=0$, which is equivalent to $\mu=0$ or
$\alpha^{2}\mu^{2}+4\alpha-4=0$. If $\mu\neq 0$, then the nonnegative root of $\alpha^{2}\mu^{2}+4\alpha-4=0$  is equal to
\begin{equation}\label{1.18}
\hat{\alpha}=\frac{2}{1+\sqrt{1+\mu^{2}}},
\end{equation}
and we have $\lambda_{1,2}=1-\hat{\alpha}$. Now, if $\mu=0$, then $|\lambda_{1,2}|=|1-\alpha|$. In this case, $\alpha=1$ is the
best choice, because in this case we have $\lambda_{1,2}=0$. On the other hand, if we set $\mu=0$  in Eq. (\ref{1.18}) we would
have $\hat{\alpha}=1$. For $\alpha<\hat{\alpha}$, the quadratic equation (\ref{1.12}) has two conjugate complex zeroes of  modulus
$1-\alpha$, which increases in modulus with decreasing $\alpha$. Thus, for the fixed eigenvalue $\mu$ of $S$, the value $\alpha$,
which minimizes the zero of largest modulus of (\ref{1.12}) is $\hat{\alpha}$. Finally, it is evident that the curve
$g({\lambda})=\pm \sqrt{{-\lambda}}\rho(S)$ is an envelope for all the curves $\pm \sqrt{{-\lambda}}\mu$,
$0\leq \mu\leq \rho(S)$, and we conclude, using the above argument, that
\[
\rho(\mathcal{G}_{\alpha^*})=\min_{0<\alpha<\frac{2}{1+\rho(S)}} \rho(\mathcal{G}_{\alpha})=1-\alpha^*,
\]
where $\alpha^*$ is defined in (\ref{1.15}). \qquad $\Box$
\end{proof}

\begin{corollary}
If $\rho(S)=0$, then according to Eq. (\ref{1.15}) we have $\alpha^*=1$, and therefore  by (\ref{1.16}) we deduce  $\rho(\mathcal{G}_{\alpha^*})=0$. This means that the method would have the highest speed of convergence. In the simplest case that $T=0$, the method converges in one iteration.
\end{corollary}

\begin{corollary}
Let $W,T \in \Bbb{R}^{n \times n}$ be symmetric positive definite and symmetric positive semidefinite matrices, respectively. Then, from Lemma 2 it can be seen that the eigenvalues of $S$ are all real and nonnegative. Moreover, from Theorem 1  the GSOR method converges if and only if
\[
0<\alpha<\frac{2}{1+\mu_{\max}(S)},
\]
where $\mu_{\max}(S)$ is the largest eigenvalue of $S=W^{-1}T$. Furthermore, by Theorem 2 by replacing $\rho(S)$ by $\mu_{\max}(S)$,  the optimal value of the relaxation parameter and the corresponding optimal convergence factor can be computed via (\ref{1.15}) and (\ref{1.16}), respectively.
\end{corollary}

It is noteworthy that, if $W$ and $T$ are  symmetric positive semidefinite and symmetric positive definite, respectively, then
the GSOR iteration method can be applied for the  equivalent real system that is obtained from $-iAx=-ib$ ($i=\sqrt{-1}$). More
generally, if there exist real numbers $\beta$ and $\delta$ such that both matrices $\widetilde{W}:= \beta W+ \delta T$ and
$\widetilde{T}:= \beta T - \delta W$ are symmetric positive semidefinite with at least one of them positive definite, we can
first multiply both sides of (\ref{1.1}) by the complex number $\beta-i\delta$ to get the equivalent system
\[
(\widetilde{W} + i\widetilde{T})x=\widetilde{b}\quad \rm{with} \quad  \widetilde{{\it b}}:=(\beta-i\delta){\it b},
\]
and then employ the GSOR iteration method to the  equivalent real
system that is obtained from the above system.

\section{Numerical experiments}\label{SEC4}

In this section, we use three examples of \cite{MHSS} and an example of \cite{LPHSS}    to illustrate the feasibility and effectiveness of the GSOR iteration method when it is employed either as a solver or as a preconditioner for GMRES to solve the equivalent real system (\ref{1.6}). In all the examples, $W$ is symmetric positive definite  and $T$ is symmetric positive semidefinite.  We also compare the performance of the GSOR method with that of the MHSS method, from point of view of both the number of iterations (denoted by IT) and the total computing times (in seconds, denoted by CPU). In each iteration of both the MHSS and GSOR iteration methods, we use the Cholesky factorization of the coefficient matrices to solve the sub-systems. The reported CPU times  are the sum of the CPU times for the convergence of the method and the CPU times for computing the Cholesky factorization. It is necessary to mention that to solve symmetric positive definite system of linear equations we have used the sparse Cholesky factorization incorporated with the symmetric approximate minimum degree reordering \cite{Saad}. To do so we have used the \verb"symamd.m" command of M{\small ATLAB} Version 7.

All the numerical experiments  were computed in double precision using some  M{\small ATLAB} codes  on a Pentium 4 Laptop, with a 2.10 GHz CPU
and 1.99GB of RAM. We use a null vector as an initial guess and the  stopping criterion
\[
\frac{\|b-Au^{(k)}\|_{2}}{\| b\|_{2}}<10^{-6},
\]
is always used where $u^{(k)}=x^{(k)}+iy^{(k)}.$


\begin{example} (See \cite{MHSS})\label{ex1}
Consider the linear system of equations (\ref{1.1}) as following
\begin{equation}\label{1.19}
\left[\left(K+\frac{3-\sqrt{3}}{\tau}I\right)+i\left(K+\frac{3+\sqrt{3}}{\tau}I\right)\right]x=b,
\end{equation}
where $\tau$ is the time step-size and $K$ is the five-point centered difference matrix approximating the negative Laplacian operator $L=-\Delta$ with homogeneous Dirichlet boundary conditions, on a uniform mesh in the unit square
$[0, 1]\times[0,1]$ with the mesh-size $h=1/(m+1)$. The matrix $K\in\mathbb{R}^{n\times n}$ possesses the tensor-product form
$K=I\otimes V_{m}+V_{m}\otimes I$, with $V_{m}=h^{-2}{\rm tridiag}(-1,2,-1)\in \mathbb{R}^{m\times m}$. Hence, $K$ is an
${n\times n}$ block-tridiagonal matrix, with $n=m^{2}$. We take
\[
W=K+\frac{3-\sqrt{3}}{\tau}I \quad {\rm and} \quad T=K+\frac{3-\sqrt{3}}{\tau}I,
\]
and the right-hand side vector $b$ with its $j$th entry $b_{j}$ being given by
\[
b_{j}=\frac{(1-i)j}{\tau(j+1)^{2}},\quad  j=1,2,\ldots,n.
\]
In our tests, we take $\tau=h$. Furthermore, we normalize coefficient matrix and right-hand side by multiplying both by $h^{2}$.
\end{example}


\begin{example}(See \cite{MHSS})\label{ex2}
Consider the linear system of equations (\ref{1.1}) as following
\[
\left[(-\omega^{2}M+K )+i(\omega C_{V}+C_{H}) \right]x=b,
\]
where $M$ and $K$ are the inertia and the stiffness matrices, $C_{V}$ and $C_{H}$ are the viscous and the hysteretic
damping matrices, respectively, and $\omega$ is the driving circular frequency. We take $C_{H}=\mu K$ with $\mu$ a damping coefficient, $M=I$, $C_{V}=10I$, and $K$ the five-point centered difference matrix approximating the negative Laplacian operator with homogeneous Dirichlet boundary conditions, on a uniform mesh in the unit square $[0, 1]\times[0, 1]$ with the mesh-size $h=1/(m+1)$. The matrix $K\in \mathbb{R}^{n\times n}$ possesses the tensor-product form $K=I\otimes V_{m}+V_{m}\otimes I$, with $V_{m}=h^{-2}{\rm tridiag}(-1,2,-1)\in \mathbb{R}^{m\times m}$. Hence, $K$ is an ${n\times n}$ block-tridiagonal matrix, with $n=m^{2}$. In addition, we set $\omega=\pi$, $\mu=0.02$, and the right-hand side vector $b$ to be $b=(1 + i)A{\bf 1}$, with ${\bf 1}$ being the vector of all entries equal to $1$. As before, we normalize the system by multiplying both sides
through by $h^{2}$.
\end{example}


\begin{example}(See \cite{MHSS})\label{ex3}
Consider the linear system of equations $(W+iT)x=b$, with
\[
T=I\otimes V+V\otimes I \quad {\rm and} \quad W=10(I\otimes V_{c}+V_{c}\otimes I)+9(e_{1}e_{m}^{T}+e_{m}e_{1}^{T})\otimes I,
\]
where $V={\rm tridiag}(-1,2,-1)\in \mathbb{R}^{m\times m}$, $V_{c}=V-e_{1}e_{m}^{T}-e_{m}e_{1}^{T}\in \mathbb{R}^{m\times m}$ and $e_{1}$  and $e_{m}$  are the first and last unit vectors in $\mathbb{R}^{m}$, respectively. We take the right-hand side vector $b$ to be $b=(1 + i)A\textbf{1}$, with $\textbf{1}$ being the vector of all entries equal to $1$.

Here $T$ and $W$ correspond to the five-point centered difference matrices approximating the negative Laplacian operator with homogeneous Dirichlet boundary conditions and periodic boundary conditions, respectively, on a uniform mesh in the unit square $[0, 1]\times[0, 1]$ with the mesh-size $h=1/(m+1)$.
\end{example}


\begin{example}\label{ex4} (See \cite{Bertaccini,LPHSS})
We consider the complex Helmholtz equation
\[
-\triangle u+\sigma_1 u + i \sigma_2 u =f,
\]
where $\sigma_1$ and $\sigma_2$ are real coefficient functions,  $u$ satisfies Dirichlet boundary conditions
in $D = [0,1] \times [0, 1]$ and $i=\sqrt{-1}$.
We discretize the problem with finite differences on a $m\times m$ grid with mesh size $h = 1/(m + 1)$.  This leads to a system of linear equations
\[
\left((K+\sigma_1 I)+i \sigma_2 I\right)x=b,
\]
where $K=I\otimes V_{m}+V_{m}\otimes I$ is the discretization of $-\triangle$  by means of centered differences, wherein $V_{m}=h^{-2}{\rm tridiag}(-1,2,-1)\in \mathbb{R}^{m\times m}$. The right-hand side vector $b$ is taken to
be $b=(1 + i)A\textbf{1}$, with $\textbf{1}$ being the vector of all entries equal to $1$.
Furthermore, before solving the system we normalize the coefficient matrix and the right-hand side vector by multiplying both by $h^{2}$. For the numerical tests we set $\sigma_1=\sigma_2=100$.
\end{example}

In Table \ref{Table1}, we  have reported the optimal values of the parameter $\alpha$ (denoted by $\alpha^{*}$)  used in both the MHSS and the GSOR iterative methods for different values of  $m$ for the four examples. The optimal  parameters $\alpha^{*}$ for the MHSS method are those presented in \cite{MHSS} (except for $m=512$). The $\alpha^{*}$ for the GSOR method is obtained from (\ref{1.15}) in which the largest eigenvalue of matrix $S$ $(\mu_{\max}(S))$  has been  estimated by a few iterations of the power method.

In Fig. 2 the optimal parameter  $\alpha^*$ for the GSOR method versus some values of  $m$ has been displayed. From Table \ref{Table1} and Fig. 1 we see that for all the examples $\alpha^*$ decreases with  the mesh-size $h$. We also see that for large values of $m$ the value $\alpha^*$ for Examples \ref{ex2} and \ref{ex4} is approximately equal to 0.455 and 0.862, respectively. For Examples \ref{ex1}
we see that $\alpha^*$ is roughly reduced by a factor of 0.96, and for Example \ref{ex3} by a factor of 0.55 as $m$ is doubled.

\begin{figure}[!hbp]
\centerline{\includegraphics[height=5cm,width=6cm]{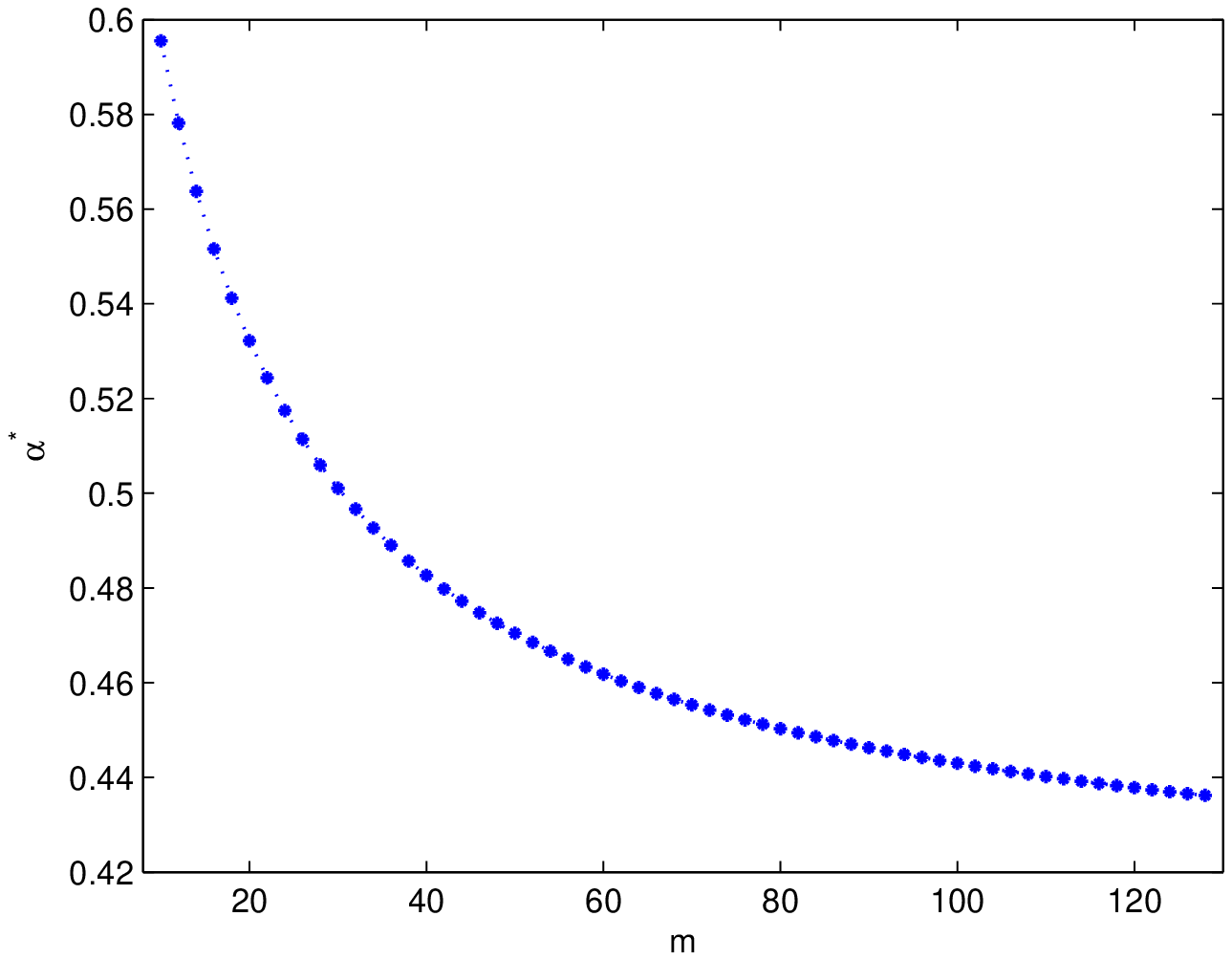}\includegraphics[height=5cm,width=6cm]{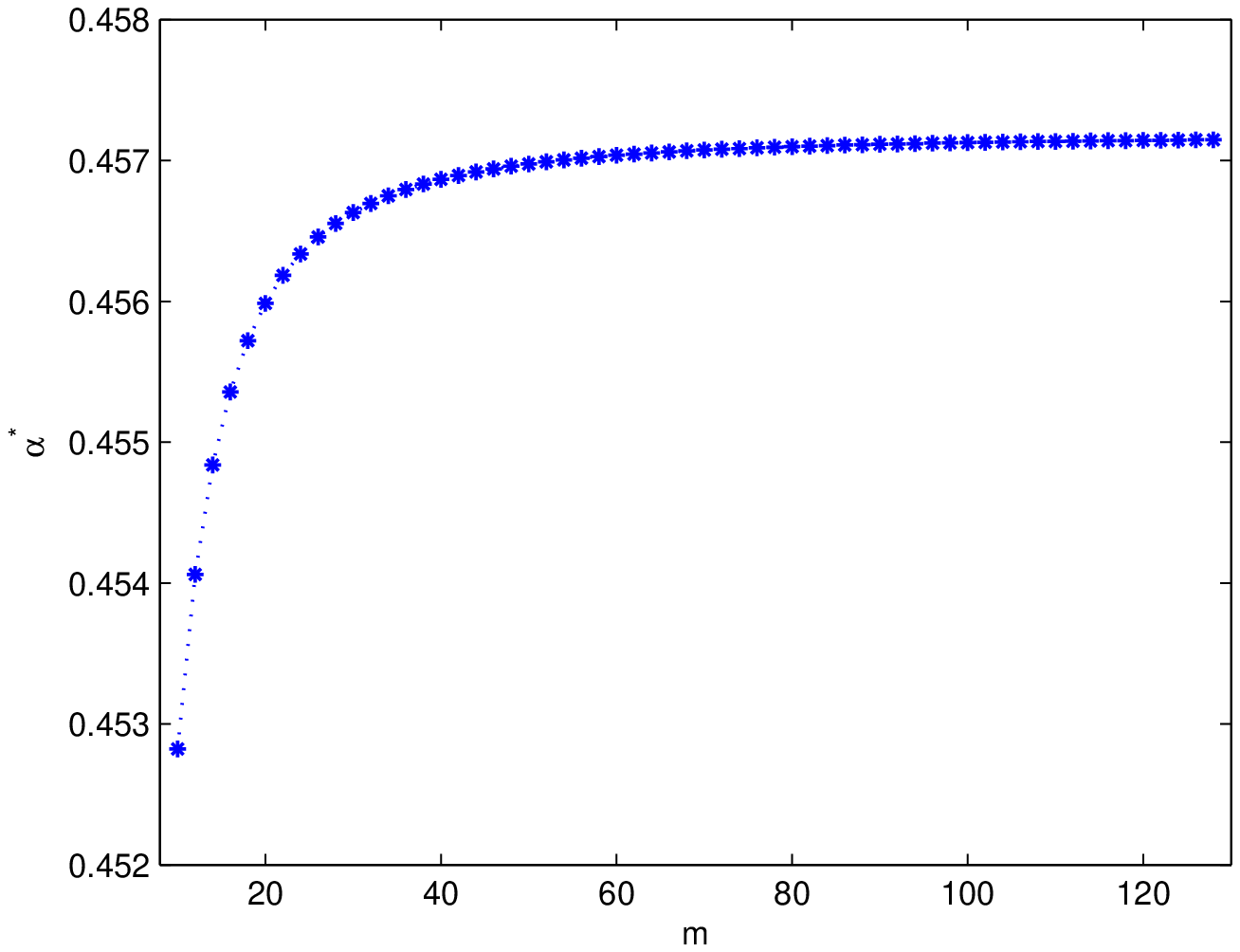}}
\centerline{\includegraphics[height=5cm,width=6cm]{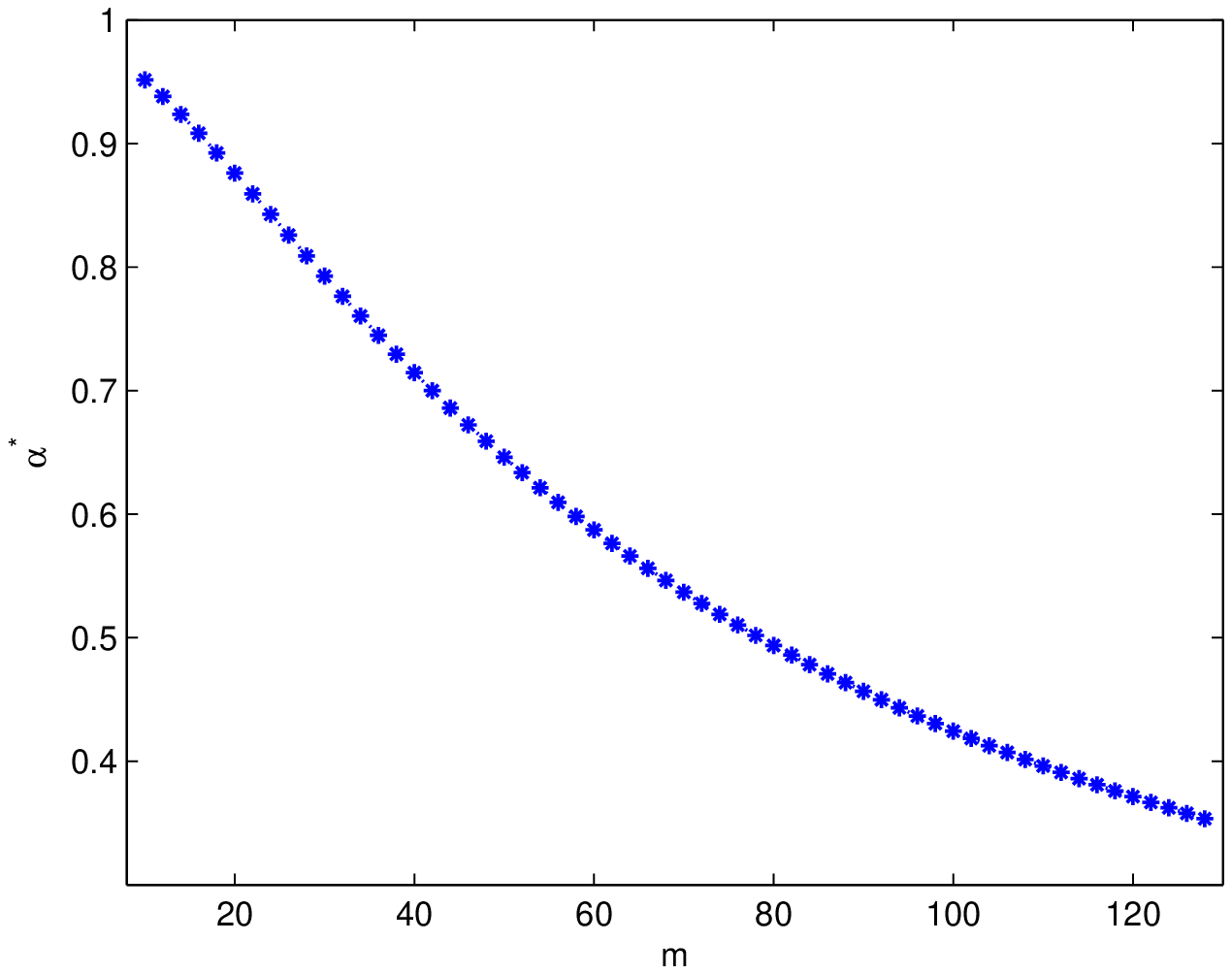}\includegraphics[height=5cm,width=6cm]{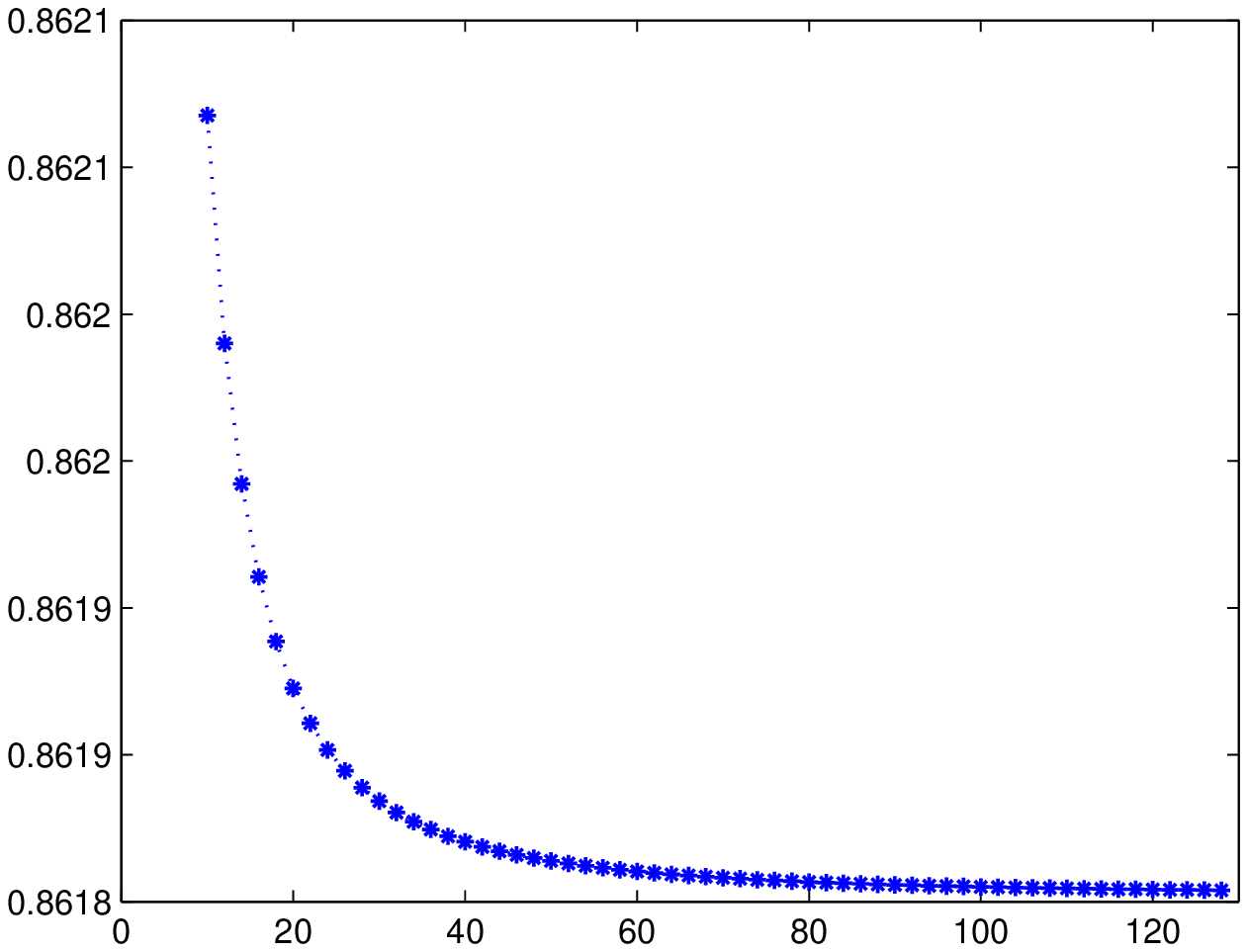}}
{\caption{The optimal parameter  $\alpha^*$ for the GSOR method versus some values of  $m$; top-left: Example 1, top-right:Example 2, down-left :Example 3, down-right: Example 4.}}
\label{Figro}
\end{figure}


 \begin{table}\label{Table1}
 \caption{ The optimal
parameters $\alpha^{*}$ for MHSS and GSOR iteration
methods.\label{Table1}} \vspace{-.2cm}
\begin{tabular}{lllllllll}\vspace{-0.2cm}\\ \hline \vspace{-0.3cm} \\  
Example   & Method  & Grid\\\cline{3-8}  \vspace{-0.3cm}
\\\vspace{0.3cm}

          &         & $16\times 16$  & $32\times 32$ & $64\times 64$ & $128\times 128$ & $256\times 256$ &
          $512\times 512$ \vspace{-0.3cm}\\  \hline \vspace{-0.3cm} \\ [0mm]

No. 1    &  MHSS   & $1.06$        & $0.75$        & $0.54$        & $0.40$    & $0.30$  & $0.21$ \\
         &  GSOR   & $0.550$       & $0.495$       & $0.457$       & $0.432$   & $0.428$ & $0.412$ \\[2mm]

No. 2    &  MHSS   & $0.21$        & $0.08$        & $0.04$        & $0.02$    & $0.01$  & $0.005$\\
         &  GSOR   & $0.455$       & $0.455$       & $0.455$       & $0.455$   & $0.455$ & $0.457$ \\[2mm]

No. 3    &  MHSS   & $1.61$        & $1.01$        & $0.53$        & $0.26$    & $0.13$  & $0.07$\\
         &  GSOR   & $0.908$       & $0.776$       & $0.566$       & $0.353$   & $0.199$ & $0.105$\\[2mm]

No. 4    &  MHSS   & $0.37$        & $0.09$        & $0.021$       & $0.005$   & $0.002$  & $0.0005$\\
         &  GSOR   & $0.862$       & $0.862$       & $0.862$       & $0.862$   & $0.862$ & $0.862$\\ \hline

\end{tabular}
\end{table}
In Figs. 3-6 we depict the eigenvalues distribution of the coefficient matrix $\mathcal{A}$ and the GSOR($\alpha^{*}$)-preconditioned matrix $\mathcal{P}_{\alpha^{*}}^{-1}\mathcal{A}$ with $m=32$, respectively, for Examples 1-4. It is evident that system which is preconditioned by GSOR method is of a well-clustered spectrum around $(1,0)$. These observations imply that when GSOR is applied as a preconditioner for GMRES, the rate of convergence can be improved considerably. This fact is further  confirmed by the numerical results presented in Tables 2-5.

\begin{figure}[!hbp]
\centerline{\includegraphics[height=5cm,width=6cm]{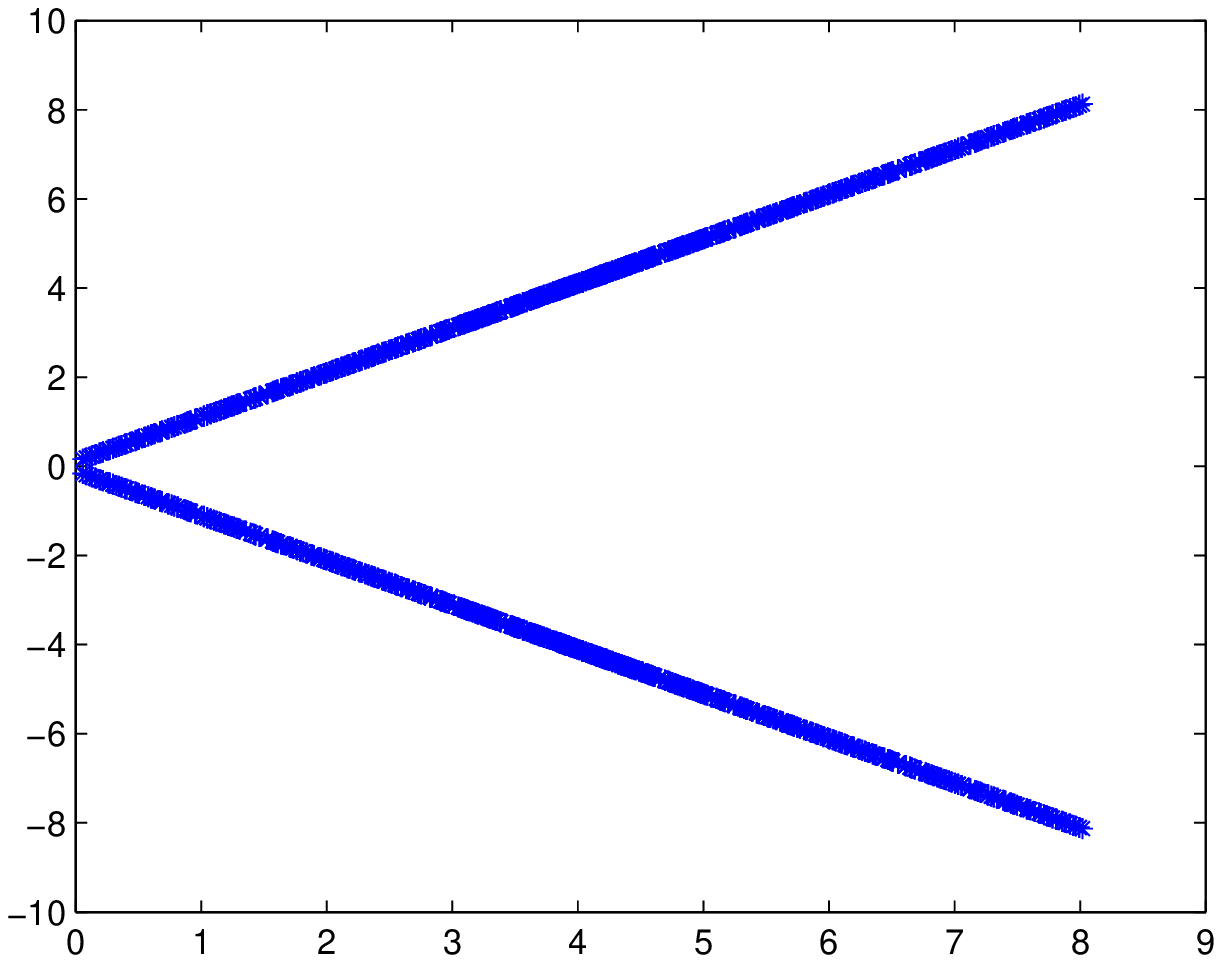}\includegraphics[height=5cm,width=6cm]{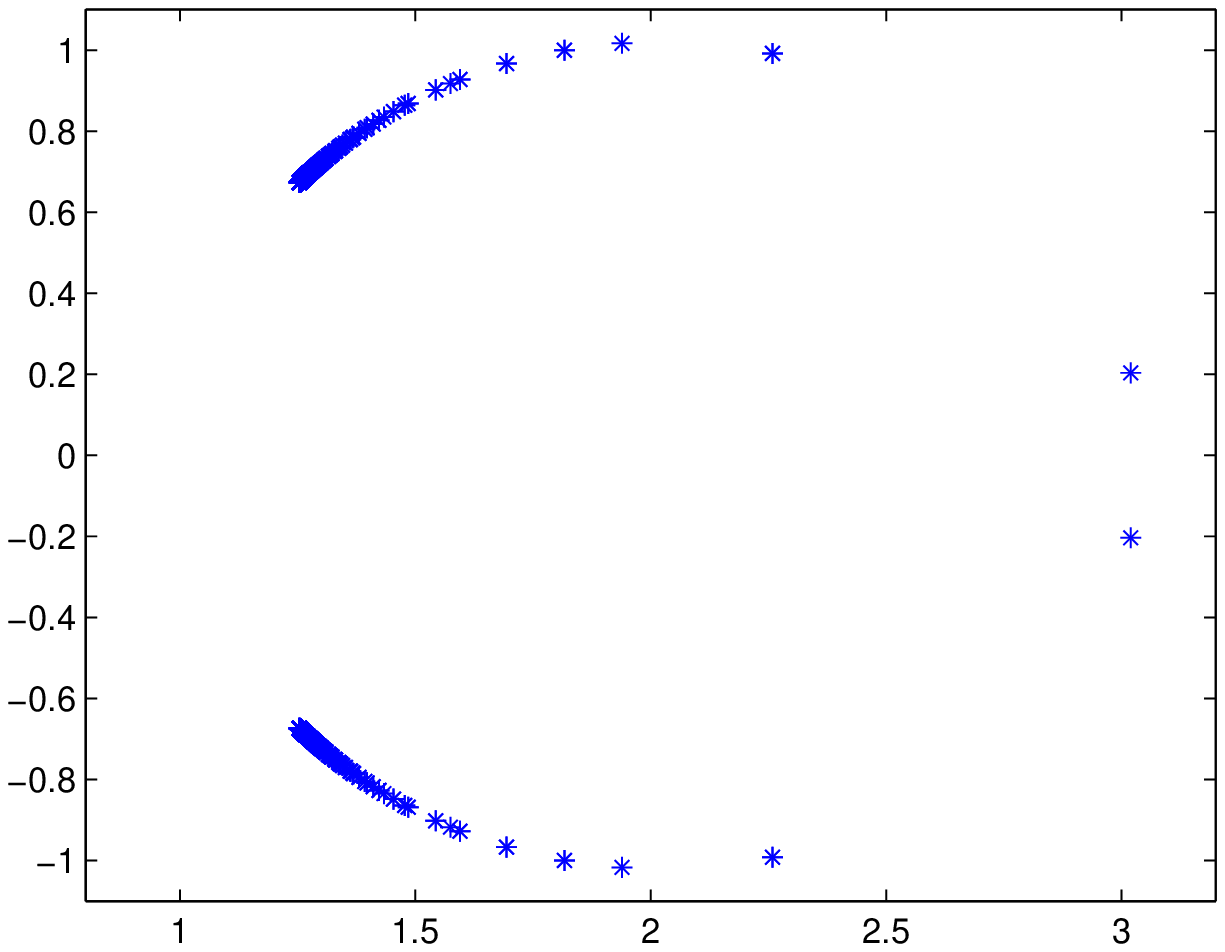}}
{\caption{Eigenvalues distribution of  the original matrix
$\mathcal{A}$ (left) and the preconditioned matrix $\mathcal
P_{\alpha^{*}}^{-1}\mathcal{A}$ (right) for Example
\ref{ex1} with $m=32$. }}\label{Fig1}\vspace{0cm}
\end{figure}

\begin{figure}[!hbp]
\centerline{\includegraphics[height=5cm,width=6cm]{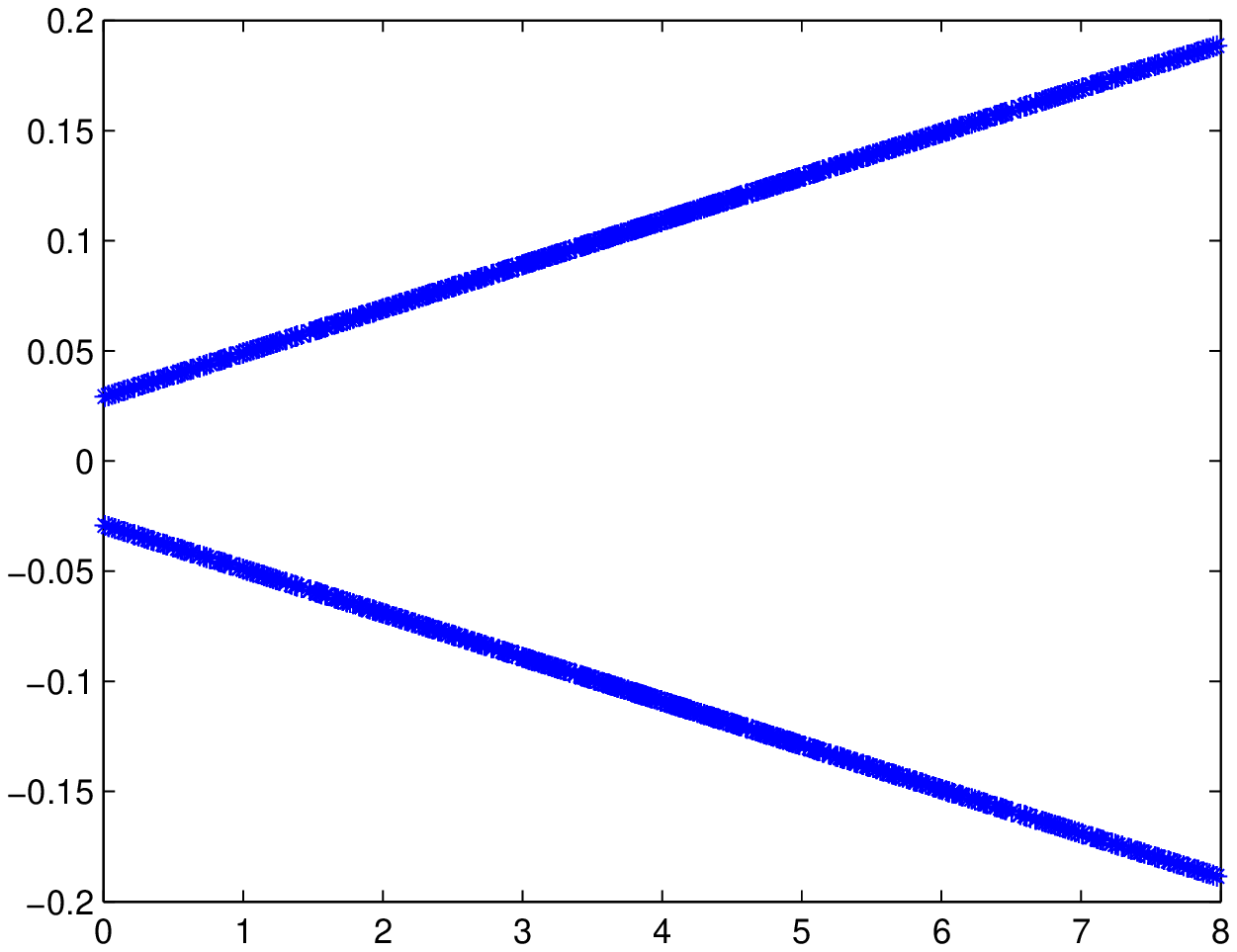}\includegraphics[height=5cm,width=6cm]{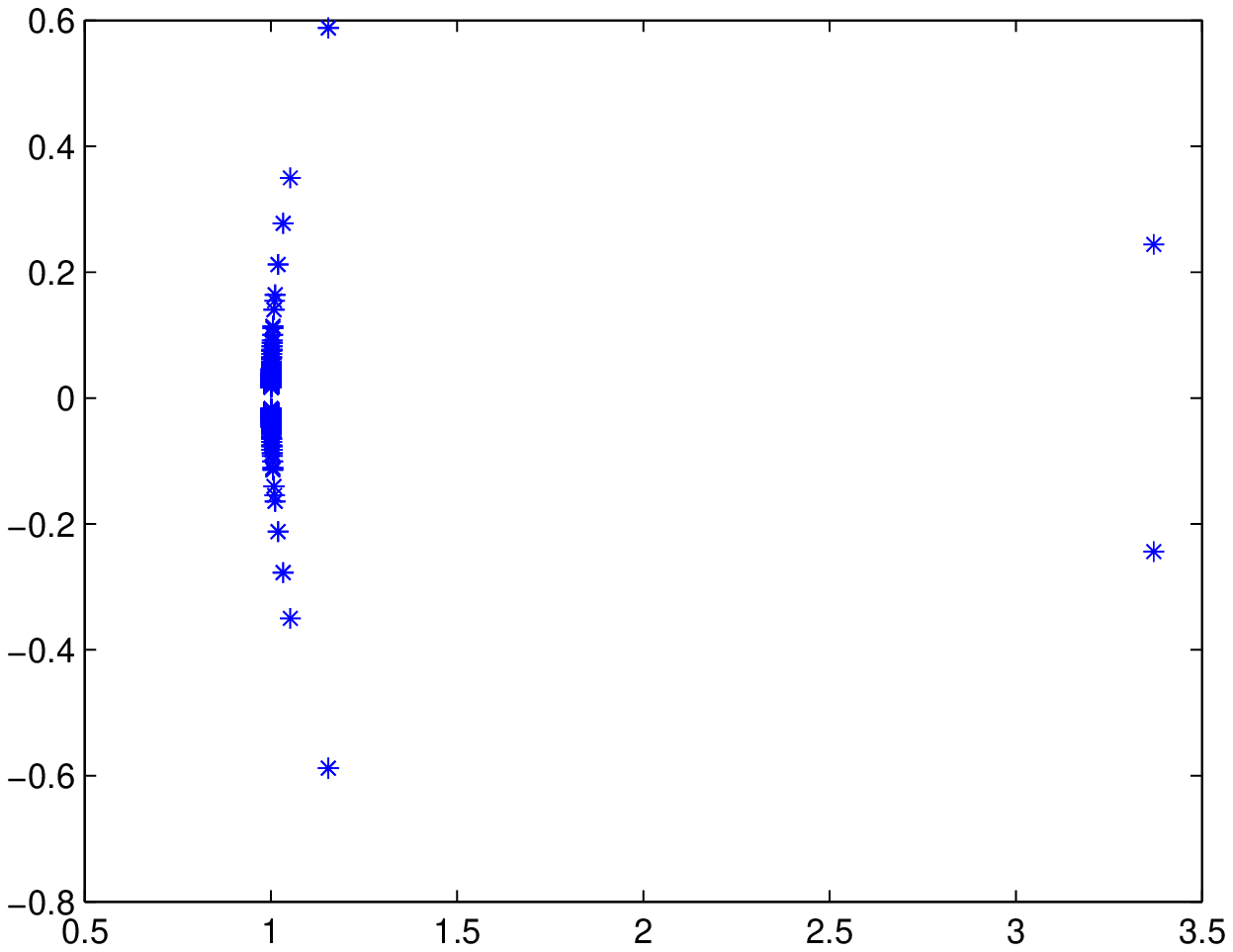}}
{\caption{Eigenvalues distribution of  the original matrix $\mathcal{A}$ (left) and the preconditioned matrix $\mathcal P_{\alpha^{*}}^{-1}\mathcal{A}$ (right) for Example \ref{ex2} with $m=32$. }}\label{Fig2}\vspace{0cm}
\end{figure}

\begin{figure}[!hbp]
\centerline{\includegraphics[height=5cm,width=6cm]{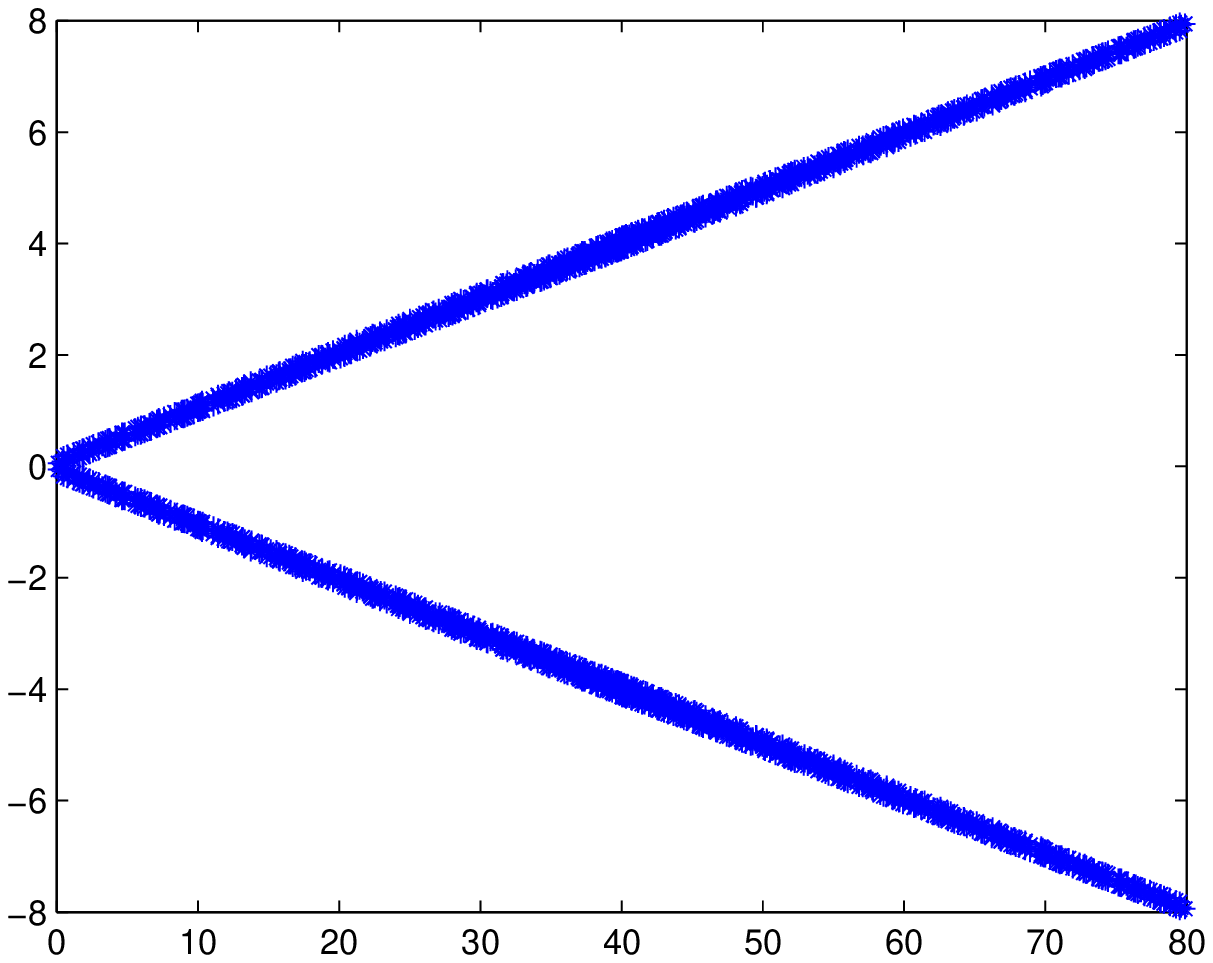}\includegraphics[height=5cm,width=6cm]{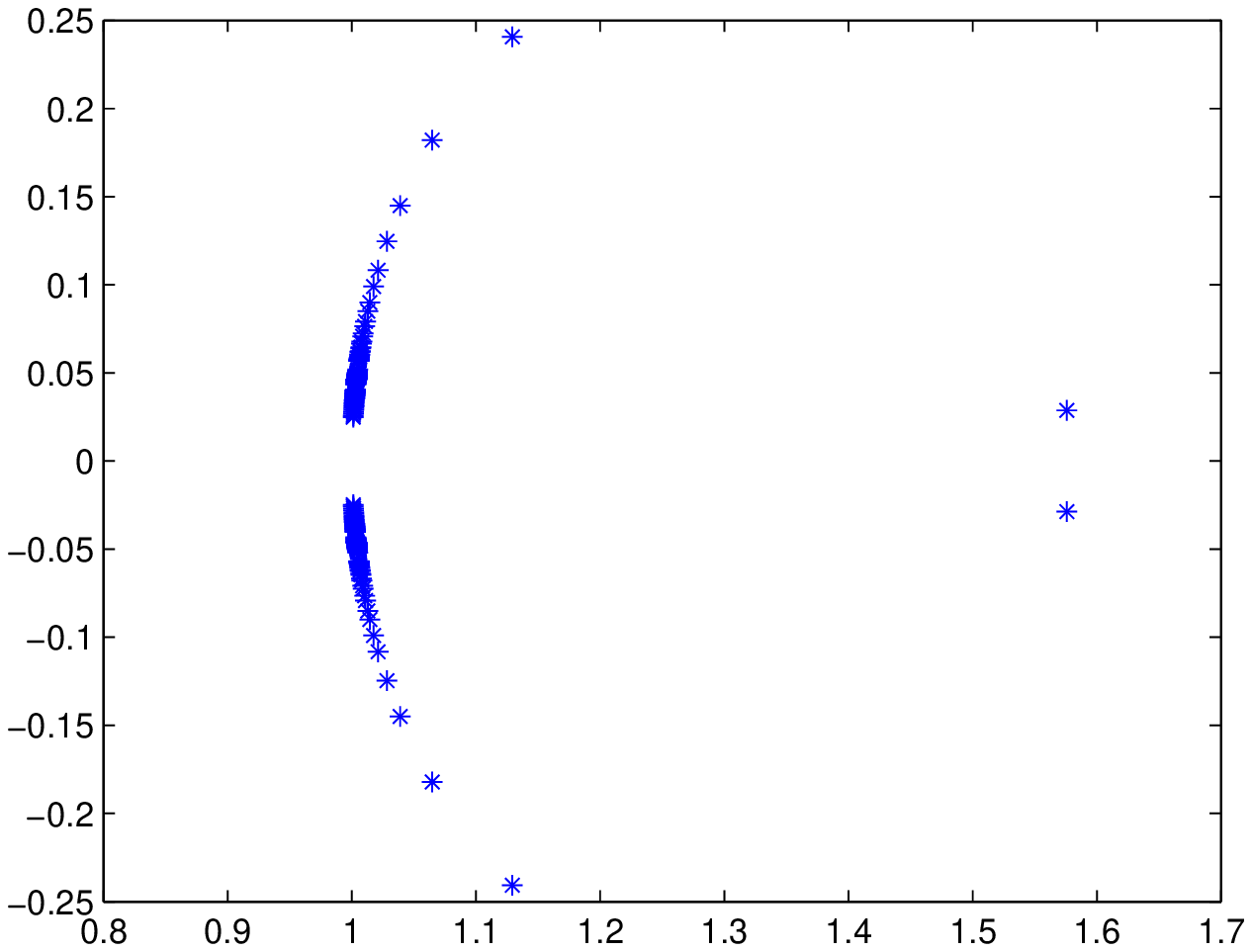}}
{\caption{Eigenvalues distribution of  the original matrix $\mathcal{A}$ (left) and the preconditioned matrix $\mathcal
P_{\alpha^{*}}^{-1}\mathcal{A}$ (right) for Example \ref{ex3} with $m=32$. }}\label{Fig3}\vspace{0cm}
\end{figure}

\begin{figure}[!hbp]
\centerline{\includegraphics[height=5cm,width=6cm]{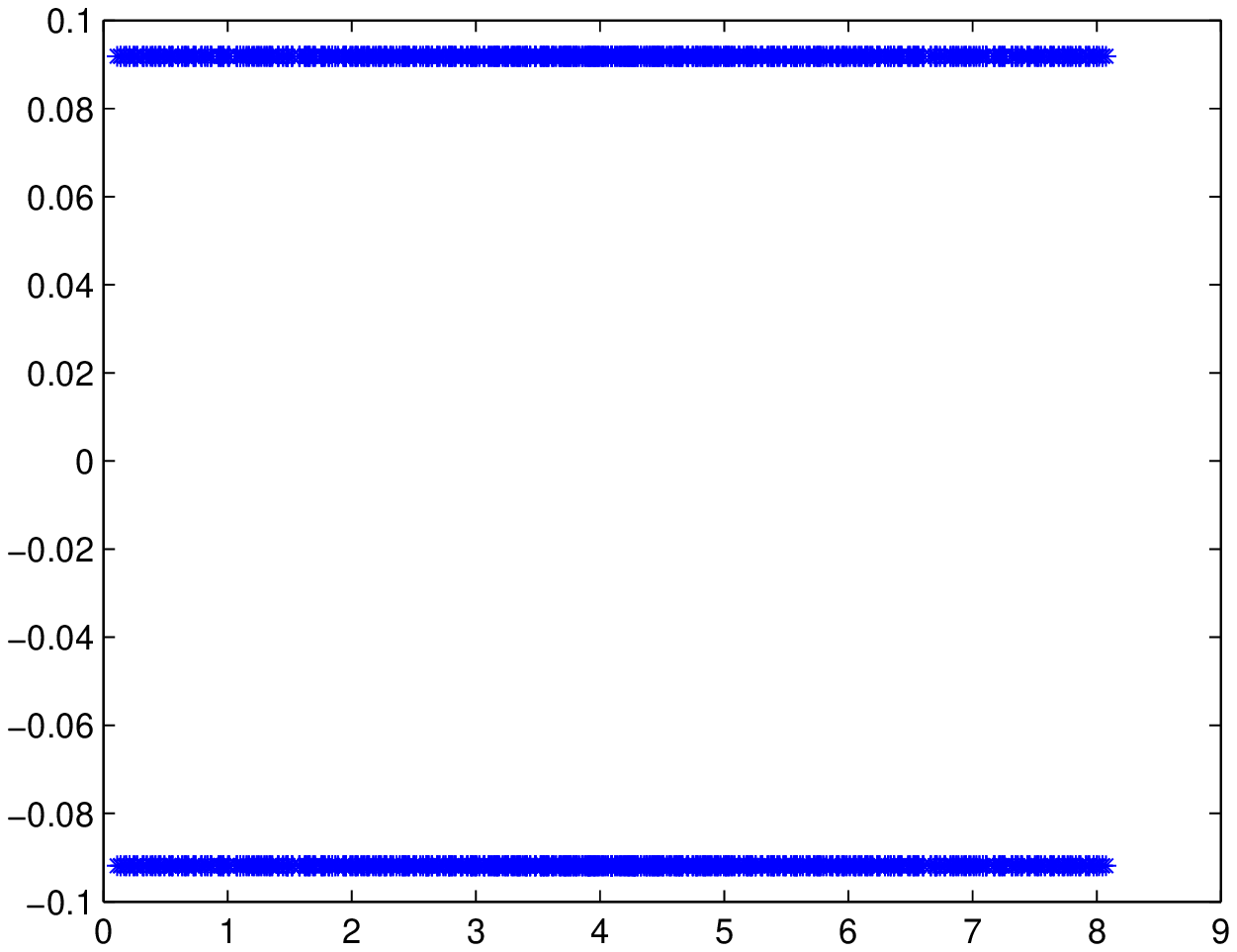}\includegraphics[height=5cm,width=6cm]{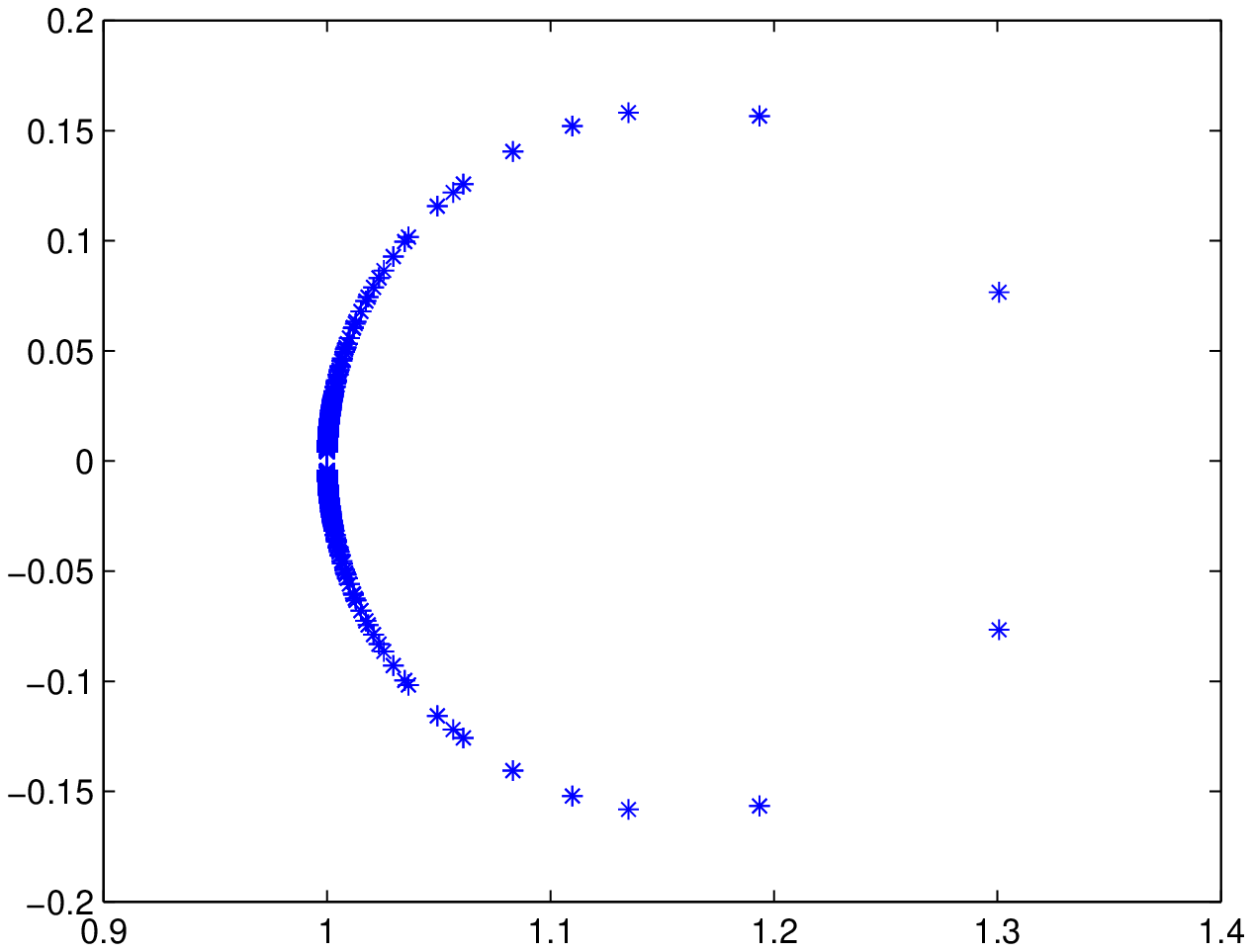}}
{\caption{Eigenvalues distribution of  the original matrix $\mathcal{A}$ (left) and the preconditioned matrix $\mathcal
P_{\alpha^{*}}^{-1}\mathcal{A}$ (right) for Example \ref{ex4} with $m=32$. }}\label{Fig4}\vspace{0cm}
\end{figure}


Numerical results for Example \ref{ex1}  are listed in Table \ref{Table2}. This table presents IT and CPU times  for the MHSS, GSOR, GMRES(10), GSOR-preconditioned GMRES(10) methods. The MHSS iteration is employed to solve the original complex system (\ref{1.1}) and the three other methods
 to solve the equivalent real system (\ref{1.6}). As seen, the GSOR method is superior to the MHSS method in terms of both iterations and CPU times. As a preconditioner, we observe that  GMRES(10) in conjunction with the GSOR method drastically reduces the number of iterations of the GMRES(10) method. Nevertheless CPU times for the GSOR-preconditioned GMRES(10) is slightly greater than those of the GSOR method.


 \begin{table}
 \caption{ Numerical results for Example \ref{ex1}. \label{Table2}}
\vspace{-.2cm}
\begin{tabular}{lllllllll}\vspace{-0.2cm}\\ \hline \vspace{-0.3cm} \\  

Method    &  $m\times m$ & $16\times 16$ & $32\times
32$&$64\times 64$ & $128\times 128$ & $256\times 256$ & $512\times 512$\vspace{-0.0cm}\\
\hline \vspace{-0.3cm} \\ \vspace{0.0cm}

MHSS              & IT  & $40$   & $54$   &$73$    & $98$    & $133$    &  181  \\\vspace{0.2cm}
                  & CPU & $0.11$ & $0.27$ & $1.53$ & $9.66$  & $65.22$  &  546.69  \\\vspace{0.0cm}

GSOR              & IT  & $19$   & $22$   &$24$    & $26$    & $27$     &  27  \\\vspace{0.2cm}
                  & CPU & $0.05$ & $0.08$ & $0.31$ & $1.47$  & $7.19$   &  44.70  \\\vspace{0.0cm}

GMRES($10$)       & IT  & $44$   & $93$   & $163$  & $288$   & $526$    &  974  \\\vspace{0.2cm}
                  & CPU & $0.08$ & $0.48$ & $3.27$ & $24.73$ & $222.66$ &  2755.23  \\\vspace{0.0cm}

GSOR-GMRES($10$)  & IT  & $3$    & $3$    & $3$    &   $4$   & $4$      &  4 \\
                  & CPU & $0.06$ & $0.13$ & $0.44$ &  $2.42$ & $10.97$  & 79.33  \\\hline

\end{tabular}
\end{table}

In Table \ref{Table3}, we show numerical results for Example
\ref{ex2}. In this table, a dagger $(\dag)$ means that the method fails to converge in 2000 iterations. As seen, the GSOR iteration method is more effective than the MHSS iteration method in terms of both iterations and CPU
times. Even with the increase of problem size, we  see that the
number of GMRES(10) iterations with the GSOR preconditioner remain
almost constant and are significantly less than those of the GMRES(10) method. Hence, the GSOR
preconditioner can significantly improve the convergence
behavior of GMRES(10).

 \begin{table}
 \caption{ Numerical results for Example \ref{ex2}. \label{Table3}}
\vspace{-.2cm}
\begin{tabular}{lllllllll}\vspace{-0.2cm}\\ \hline \vspace{-0.3cm} \\  

Method    &  $m\times m$ & $16\times 16$  & $32\times 32$  &$64\times 64$ & $128\times 128$ &
$256\times 256$ & $512\times 512$ \vspace{-0.0cm}\\  \hline \vspace{-0.3cm} \\

MHSS              & IT  & $34$    & $38$    &$50$     &$81$     & $139$   & 250 \\\vspace{0.2cm}
                  & CPU & $0.06$  & $0.20$  & $1.28$  & $8.032$ & $68.25$ & 746.49 \\\vspace{0.0cm}

GSOR              & IT  & $26$    & $24$    &$24$     & $23$    & $23$    & 23 \\\vspace{0.2cm}
                  & CPU & $0.05$  & $0.08$  & $0.33$  & $1.31$  & $6.23$  & 38.83 \\\vspace{0.0cm}

GMRES($10$)       & IT  & $23$    & $117$   & $228$   & $670$   & $\dag$  & $\dag$ \\   \vspace{0.2cm}
                  & CPU & $0.08$  & $0.64$  & $4.53$  & $58.88$ & $--$    & $--$ \\    \vspace{0.0cm}

GSOR-GMRES($10$)  & IT  & $2$     & $2$     & $2$     & $2$     & $2$     & 2 \\
                  & CPU & $0.06$  & $0.11$  & $0.31$  & $1.31$  & $6.25$  & 37.14 \\\hline

\end{tabular}
\end{table}

Numerical results for Example \ref{ex3} are presented in Table \ref{Table4}. In terms of the iteration steps, GSOR-preconditioned GMRES(10) performs much better than GSOR, MHSS and GMRES(10). In  terms of computing times,  GSOR-preconditioned GMRES(10) costs less CPU than MHSS and  GMRES(10) and also less CPU than GSOR expect  for $m=32$ and $m=64$. Hence, we find that as a preconditioner for GMRES(10), the GSOR is of  high performance, especially when problem size increases.


 \begin{table}
 \caption{ Numerical results for Example \ref{ex3}. \label{Table4}}
\vspace{-.2cm}
\begin{tabular}{llllllll}\vspace{-0.2cm}\\ \hline \vspace{-0.3cm} \\  

Method    &  $m\times m$ & $16\times 16$  & $32\times 32$ &$64\times 64$ & $128\times 128$ &
$256\times 256$ & $512\times 512$ \vspace{-0.0cm}\\  \hline \vspace{-0.3cm} \\

MHSS              & IT  & $53$    & $76$   &$130$   &$246$    & $468$    & $869$ \\\vspace{0.2cm}
                  & CPU & $0.06$  & $0.39$ & $3.01$ & $26.72$ & $245.26$ & $2792.10$ \\\vspace{0.0cm}

GSOR              & IT  & $7$     & $11$   &$20$    & $35$    & $71$     & $131$ \\\vspace{0.2cm}
                  & CPU & $0.06$  & $0.07$ & $0.31$ & $2.28$  & $20.39$  & $241.49$ \\\vspace{0.0cm}

GMRES($10$)       & IT  & $19$    & $49$   & $91$   & $316$   & $1081$   & $\dag$ \\\vspace{0.2cm}
                  & CPU & $0.06$  & $0.25$ & $1.88$ & $27.11$ & $459.95$ & $--$ \\\vspace{0.0cm}

GSOR-GMRES($10$)  & IT  & $2$     & $2$    & $2$    & $3$     & $4$      & $8$ \\
                  & CPU & $0.05$  & $0.09$ & $0.34$ & $2.16$  & $13.73$  & $161.14$ \\\hline

\end{tabular}
\end{table}


 \begin{table}
 \caption{Numerical results for Example \ref{ex4}. \label{Table5}}
\vspace{-.2cm}
\begin{tabular}{llllllll}\vspace{-0.2cm}\\ \hline \vspace{-0.3cm} \\  

Method    &  $m\times m$ & $16\times 16$  & $32\times 32$ &$64\times 64$ & $128\times 128$ &
$256\times 256$ & $512\times 512$ \vspace{-0.0cm}\\  \hline \vspace{-0.3cm} \\

MHSS              & IT  & $30$    & $36$   &$39$    &$40$     & $41$     & $41$ \\\vspace{0.2cm}
                  & CPU & $0.05$  & $0.14$ & $0.55$ & $2.69$  & $16.5$   & $80.95$ \\\vspace{0.0cm}

GSOR              & IT  & $8$     & $8$    &$8$     & $8$     & $7$      & $7$ \\\vspace{0.2cm}
                  & CPU & $0.03$  & $0.06$ & $0.11$ & $0.531$  & $2.50$  & $17.63$ \\\vspace{0.0cm}

GMRES($10$)       & IT  & $5$     & $12$   & $24$   & $66$    & $219$    & $761$ \\\vspace{0.2cm}
                  & CPU & $0.05$  & $0.13$ & $0.36$ & $4.55$ & $99.53$   & $1616.88$ \\\vspace{0.0cm}

GSOR-GMRES($10$)  & IT  & $2$     & $2$    & $2$    & $2$     & $2$      & $2$ \\
                  & CPU & $0.04$  & $0.08$ & $0.31$ & $1.52$  & $8.48$  & $42.67$ \\\hline

\end{tabular}
\end{table}

In Table \ref{Table5}, numerical results of Example \ref{ex4} are presented. All of the comments and observations which we have given for the previous examples can also be posed here. As a preconditioner we see that the GSOR preconditioner drastically reduces the iteration numbers of the GMRES(10) method. For example, the GMRES(10) converges in 761 iterations, while the GMRES(10) in conjunction with the GSOR preconditioner converges only in 2 iterations. In addition, the MHSS method can not compete with the GSOR method in terms of iterations and CPU times.

\section{Conclusion}\label{SEC5}

In this paper we have utilized the generalized successive overrelaxation (GSOR) iterative method to solve the equivalent real formulation of complex linear system (\ref{1.6}), where $W$ is symmetric positive definite and $T$ is symmetric positive semidefinite. Convergence properties of the method have been also investigated. Besides its use as a solver, the GSOR iteration has also been used as a preconditioner to accelerate Krylov subspace methods such as GMRES. Some numerical have been presented to show the effectiveness of the method. Our numerical examples show that our method is quite suitable for such problems. Moreover, the presented numerical experiments show that the GSOR method is superior to MHSS in terms of the iterations and CPU times.

\section*{Acknowledgments}

The authors  are grateful to the anonymous referees and the editor of the journal for their valuable comments and suggestions.






\end{document}